\documentclass[a4paper,12pt]{article}
\usepackage{amssymb,amsmath,amsthm,amsfonts,latexsym}
\usepackage{}
\usepackage{graphicx}
\usepackage[pdftex,bookmarks,colorlinks=false]{hyperref}
\usepackage[hmargin=1.25in,vmargin=1.25in]{geometry}
\usepackage[auth-sc-lg,affil-sl]{authblk}
\usepackage{fancyhdr}

\newtheorem{theorem}{Theorem}[section]

\newtheorem{definition}[theorem]{Definition}

\newtheorem{lemma} [theorem]{Lemma}

\newtheorem{proposition}[theorem]{Proposition}

\title{\bf {\sc A Creative Review on Integer Additive Set-Valued Graphs}}

\author{N. K. Sudev}
\affil{\small Department of Mathematics\\ Vidya Academy of Science \& Technology\\Thalakkottukara, Thrissur-680501, Kerala, India.\\E-mail:{\em sudevnk@gmail.com}}

\author{K. A. Germina}
\affil{\small PG \& Research Department of Mathematics\\ Marymatha Arts \& Science College\\Mananthavady, Wayanad-670645, Kerala, India.\\E-mail:{\em srgerminaka@gmail.com}}

\author{K. P. Chithra} 
\affil{\small Naduvath Mana, Nandikkara \\Thrissur-680301, Kerala, India.\\E-email:{\em chithrasudev@gmail.com}}

\date{}

\begin{document}
\maketitle

\begin{abstract}
For a non-empty ground set $X$, finite or infinite, the {\em set-valuation} or {\em set-labeling} of a given graph $G$ is an injective function $f:V(G) \to \mathcal{P}(X)$, where $\mathcal{P}(X)$ is the power set of the set $X$. A set-indexer of a graph $G$ is an injective set-valued function $f:V(G) \to \mathcal{P}(X)$ such that the function $f^{\ast}:E(G)\to \mathcal{P}(X)-\{\emptyset\}$ defined by $f^{\ast}(uv) = f(u ){\ast} f(v)$ for every $uv{\in} E(G)$ is also injective, where $\ast$ is a binary operation on sets. An integer additive set-indexer is defined as an injective function $f:V(G)\to \mathcal{P}({\mathbb{N}_0})$ such that the induced function $f^+:E(G) \to \mathcal{P}(\mathbb{N}_0)$ defined by $f^+ (uv) = f(u)+ f(v)$ is also injective, where $\mathbb{N}_0$ is the set of all non-negative integers. In this paper, we critically and creatively review the concepts and properties of integer additive set-valued graphs.
\end{abstract}
\textbf{Key words}: Set-labeling, set-indexer, integer additive set-labeling, integer additive set-indexer. 

\vspace{0.04in}
\noindent \textbf{AMS Subject Classification : 05C78}

\section{Introduction}

For all  terms and definitions, not defined specifically in this paper, we refer to \cite{FH}. Unless mentioned otherwise, all graphs considered here are simple, finite and have no isolated vertices.

The researches on graph labeling problems commenced with the introduction of the concept of number valuations of graphs in \cite{AR1}. Since then, the studies on graph labeling have contributed significantly to the researches in graph theory and associated fields. Graph labeling problems have numerous theoretical and practical applications. Many types of graph labelings are surveyed and listed in \cite{JAG}.

Motivated from certain problems related to social interactions and social networks, Acharya introduced, in \cite{BDA1}, the notion of set-valuation of graphs analogous to the number valuations of graphs. For a non-empty ground set $X$, finite or infinite, the {\em set-valuation} or {\em set-labeling} of a given graph $G$ is an injective function $f:V(G) \to \mathcal{P}(X)$, where $\mathcal{P}(X)$ is the power set of the set $X$. 

Also, in \cite{BDA1}, a {\em set-indexer} of a graph $G$ is defined as an injective set-valued function $f:V(G) \to \mathcal{P}(X)$ such that the function $f^{\ast}:E(G)\to \mathcal{P}(X)-\{\emptyset\}$ defined by $f^{\ast}(uv) = f(u ){\ast} f(v)$ for every $uv{\in} E(G)$ is also injective, where $\mathcal{P}(X)$ is the set of all subsets of $X$ and $\ast$ is a binary operation on sets. 

Taking the symmetric difference of two sets as the operation  between two set-labels of the vertices of $G$, the following theorem was proved in \cite{BDA1}.

\begin{theorem}\label{T-SIG}
\cite{BDA1} Every graph has a set-indexer.
\end{theorem}

Let $A$ and $B$ be two sets. The {\em sum set} of $A$ and $B$ is denoted by $A+B$ and is defined as $A+B=\{a+b : a\in A, b\in B\}$. For any positive integer $n$, an $n$-integral multiple of a set $A$, denoted by $n.A$ and is defined as $n.A=\{na : a\in A\}$. We denote the cardinality of a set $A$ by $|A|$.

In this work, we review the studies on the graphs which admit a particular type of set-labeling called integer additive set-labeling, defined in terms of the sum sets of given sets of non-negative integers.

\section{On Integer Additive Set-Labeled Graphs}

During a personal communication with the second author, Acharya introduced the notion of integer additive set-labeling, using the concept of sum sets of two sets of non-negative integers, as follows.

\begin{definition}{\rm 
Let $\mathbb{N}_0$ be the set of all non-negative integers and let $\mathcal{P}(\mathbb{N}_0)$ be its power set. A function $f:V(G)\to \mathcal{P}(\mathbb{N}_0)$, whose associated function $f^+(uv):E(G)\to \mathcal{P}(\mathbb{N}_0)$ is defined by $f^+(uv)=f(u)+f(v)$, is said to be an {\em integer additive set-labeling} if it is injective. A graph $G$ which admits an integer additive set-labeling is called an {\em integer additive set-labeled graph} or {\em integer additive set-valued graph}. }
\end{definition}

The notion of integer additive set-indexers of graphs was first appeared in \cite{GA} as follows. 

\begin{definition}\label{DEFN-1}{\rm
An {\em integer additive set-indexer} (IASI) is defined as an injective function $f:V(G)\to \mathcal{P}(\mathbb{N}_0)$ such that the induced function $f^+:E(G) \to \mathcal{P}(\mathbb{N}_0)$ defined by $f^+ (uv) = f(u)+ f(v)=\{a+b: a \in f(u), b \in f(v)\}$ is also injective. A graph $G$ which admits an IASI is called an integer additive set-indexed graph.}
\end{definition} 

In this discussion, we denote the associated function of an IASL $f$, defined over the edge set of $G$ by $f^+$.

\begin{definition}\label{DEFN-1a}{\rm
An IASL (or IASI) is said to be {\em $k$-uniform} if $|f^+(e)| = k$ for all $e\in E(G)$. That is, a connected graph $G$ is said to have a $k$-uniform IASL if all of its edges have the same set-indexing number $k$.}
\end{definition}

If either $f(u)$ or $f(v)$ is countably infinite, then their sum set will also be countably infinite and hence the study of the cardinality of the sum set $f(u)+f(v)$ becomes countably infinite. Hence, we restrict our discussion to finite sets of non-negative integers.  

Certain studies about integer additive set-indexed graphs have been done in \cite{TMKA}, \cite{GS1}, \cite{GS2} and \cite{GS0}. The following are the important terms and definitions made in these papers.

The cardinality of the set-label of an element (vertex or edge) of a graph $G$ is called the {\em set-indexing number} of that element. If the set-labels of all vertices of $G$ have the same cardinality, then the vertex set $V(G)$ is said to be {\em uniformly set-indexed}. A graph is said to admit {\em $(k,l)$-completely uniform IASI (or IASL)} $f$  if $f$ is a $k$-uniform IASL (or IASI) and $V(G)$ is $l$-uniformly set-indexed.

Analogous to Theorem \ref{T-SIG}, we have proved the following theorem.

\begin{theorem}\label{T-IASL1}
Every graph has an integer additive set-labeling.
\end{theorem}

The proof follows from the fact that the sum set of two sets of non-negative integers is again a set of non-negative integers. 

The following theorem was proved in \cite{GS0} by choosing the set-labels of the vertices of a given graph $G$ in a suitable manner.

\begin{theorem}\label{T-IASI1a}
\cite{GS0} Every graph has an integer additive set-indexer.
\end{theorem}

The sum sets of different pairs of sets non-negative integers need not be distinct in all cases. But we can establish the existence of an IASI for any given graph $G$ if we choose the set-labels of the vertices of $G$ suitably in such a way that no two edges of $G$ have the same set-indexing number. 

\section{IASLs of Certain Graph Operations}

We know that if $A$ and $B$ are two sets of non-negative integers, then their sum set $A+B$ is also a set of non-negative integers. That is, if $A,B\subset \mathbb{N}_0$, then $A+B\subset \mathbb{N}_0$. This property leads us to the following results on certain operations and products of IASL graphs.

\noindent First, we consider the union of two IASL graphs.

\begin{theorem}
The union of two IASL (or IASI) graphs also admits (induced) IASL (or IASI).
\end{theorem}
\begin{proof}
Let $G_1$ and $G_2$ be two graphs that admit the IASLs $f_1$ and $f_2$ respectively, chosen in such a way that  $f_1$ and $f_2$ are the same for the common vertices of $G_1$ and $G_2$, if any. Then, define a function $f$ on $G_1\cup G_2$, such that $f=f_1$ for all vertices in $G_1$ and $f=f_2$ for all vertices in $G_2$ such that the associated function $f^+$ of $f$ is defined as $f^+(uv)=f_1(u)+f_1(v)$ if $uv\in E(G_1)$ and $f^+(uv)=f_2(u)+f_2(v)$ if $uv\in E(G_2)$. Clearly, $f$ is an IASL of $G_1\cup G_2$.
\end{proof} 

\noindent Next, in the following theorem, we check whether the join of two IASL graphs is an IASL graph.

\begin{theorem}
The join of two IASL (or IASI) graphs also admits (induced) IASL (or IASI).
\end{theorem}
\begin{proof}
Let $G_1$ and $G_2$ be two graphs that admit the IASLs $f_1$ and $f_2$ respectively. Consider the join $G=G_1+G_2$. Define a function $f$ on $G$ in such a way that $f=f_1$ for all vertices in $G_1$ and $f=f_2$ for all vertices in $G_2$ and the associated function $f^+:E(G)\to \mathcal{P}(\mathbb{N}_0)$ of $f$ is defined as 
\begin{equation}
f^+(uv)=
\begin{cases}
f_1(u)+f_1(v) & ~\text{if} ~ u,v\in V(G_1)\\
f_2(u)+f_2(v) & ~\text{if} ~ u,v\in V(G_2)\\
f_1(u)+f_2(v) & ~\text{if} ~ u\in V(G_1), v\in V(G_2).
\end{cases}
\end{equation}
Clearly, $f$ is an IASL of $G_1+G_2$.
\end{proof}

\begin{theorem}
The complement of an IASL graph is also an IASL graph.
\end{theorem}
\begin{proof}
The proof follows from the fact that a graph $G$ and its complement $\overline{G}$ have the same set of vertices and hence have the same set of set-labels.  That is, an IASL $f$ defined on a graph $G$ is an IASL of the complement of $G$ also.
\end{proof}

\noindent Now, recall the definitions of the certain fundamental products of two graphs given in \cite{HIS}.

Let $G_1(V_1,E_1)$ and $G_2(V_2,E_2)$ be two graphs. Then, the {\em Cartesian product} of $G_1$ and $G_2$, denoted by $G_1\Box G_2$, is the graph with vertex set $V_1\times V_2$  defined as follows. Let $u=(u_1, u_2)$ and $v=(v_1,v_2)$ be two points in $V_1\times V_2$. Then, $u$ and $v$ are adjacent in $G_1\Box G_2$ whenever [$u_1=v_1$ and $u_2$ is adjacent to $v_2$] or [$u_2=v_2$ and $u_1$ is adjacent to $v_1$].

The {\em direct product} of two graphs $G_1$ and $G_2$, is the graph whose vertex set is $V(G_1)\times V(G_2)$ and for which the vertices $(u,v)$ and $(u',v')$ are adjacent if $uu'\in E(G_1)$ and $vv'\in E(G_2)$. The direct product of $G_1$ and $G_2$ is denoted by $G_1\times G_2$.

The {\em strong product} of two graphs $G_1$ and $G_2$ is the graph, denoted by $G_1\boxtimes G_2$, whose vertex set is $V(G_1) \times V(G_2)$ and for which the vertices $(u,v)$ and $(u',v')$ are adjacent if $[uu'\in E(G_1)~\text{and}~ v=v']$ or $[u=u' ~ \text{and}~vv'\in E(G_2)]$ or $[uu'\in E(G_1)$ and $vv'\in E(G_2)]$.

\cite{IK} The {\em lexicographic product} or {\em composition} of two graphs $G_1$ and $G_2$ is the graph, denoted by $G_1\circ G_2$, is the graph whose vertex set $V(G_1)\times V(G_2)$ and for two vertices $(u,v)$ and $(u',v')$ are adjacent if $[uu'\in E(G_1)]$ or $[u=u'~\text{and}~ vv'\in E(G_2)]$.

\noindent Now we proceed to verify whether these products of two IASL graphs admit IASLs.

\begin{theorem}
The Cartesian product of two IASL graphs also admits an IASL.
\end{theorem}
\begin{proof}
Let $u_1,u_2,\ldots, u_m$ be the vertices of $G_1$ and $v_1,v_2,\ldots, v_n$ be the vertices of $G_2$. Also let $f_1$, $f_2$ be the IASLs defined on $G_1$ and $G_2$ respectively. Now, define a function $f$ on $G_1\Box G_2$ such that $f(u_i,v_j)=f_1(u_1)+f_2(v_j)$ where the associated function $f^+:(E(G_1\Box G_2))\to \mathcal{P}(\mathbb{N}_0)$ is defined by $f^+((u_i,v_j),(u_r,v_s))=f(u_i,v_j)+f(u_r,v_s)$.
Therefore, $f$ is an IASL on $G_1\Box G_2$.
\end{proof}

\noindent Using the definition of $f$ for the vertices and the associated function $f^+$ for the edges, we can establish the existence of IASLs for the other fundamental graph products also as stated in the following results.

\begin{theorem}
The direct product of two IASL graphs also admits an IASL.
\end{theorem}

\begin{theorem}
The strong product of two IASL graphs also admits an IASL.
\end{theorem}

\begin{theorem}
The lexicographic product of two IASL graphs is also an IASL graph.
\end{theorem}

\noindent Next, we consider certain products of graphs in which we take finite number copies of any one of the two graphs and attach these copies to the vertices of the other graph using certain rules. 

The most popuar one among such graph products, is the {\em corona} of two graphs which is defined in \cite{FFH} as follows.
 
The {\em corona} of two graphs $G_1$ and $G_2$, denoted by $G_1\odot G_2$, is the graph obtained by taking $|V(G_1)|$ copies of $G_2$ and then joining the $i$-th point of $G_1$ to every point of the $i$-th copy of $G_2$.

Another commonly used graph product in various literature, is the {\em rooted product} of two graphs, which is defined in \cite{GM} as follows.

The {\em rooted product} of a graph $G_1$ on $n_1$ vertices and rooted graph $G_2$ on $n_2$ vertices, denoted by $G_1\circ G_2$, is defined as the graph obtained by taking $n_1$ copies of $G_2$, and for every vertex $v_i$ of $G_1$ and identifying $v_i$ with the root node of the $i$-th copy of $G_2$.

The admissibility of an IASL by these products of two IASL graphs is established in the following theorem.

\begin{theorem}
The corona of two IASL graphs also admits an IASL.
\end{theorem}
\begin{proof}
Let $u_1, u_2,\ldots, u_m$ be the vertices of $G_1$ and let $v_1, v_2,\ldots, v_n$ be the vertices of $G_2$. For $1\le i\le m$ and $1\le j\le n$, let $v_{ij}$ be the $j$-th vertex of the $i$-th copy of $G_2$. Now, label the vertex $v_{ij}$ of $i$-th copy $G_2$ corresponding to vertex $v_j$ of $G_2$ by the set $f_i(v_{ij})=i.f(v_j)$. Let the set-label of an edge in $G$ be the sum set of the set-labels of its end vertices. Clearly, this labeling is an IASL on $G_1\odot G_2$.
\end{proof}

\begin{theorem}
The rooted product of two IASL graphs is also an IASL graph.
\end{theorem}
\begin{proof}
In the rooted product $G_1\circ G_2$, for $1\le i\le |V(G_1)|$, the roots of the $i$-th copy of $G_2$ is identified with the $i$-th vertex of $G_1$. Label these identified vertices by the corresponding set-labels of $G_1$. All other vertices in the different copies of $G_2$ by distinct integral multiples of the set-labels of corresponding vertices of $G_2$. Also, let the set-labels of edges in $G_1\circ G_2$ be the sum set of the set-labels of their end vertices. Clearly, this set-labeling is an IASL on $G_1\circ G_2$.
\end{proof}

\section{IASLs of Certain Associated Graphs}

An IASL of a graph $G$ induces an IASL to certain other graphs associated to $G$. The induction property of a IASL of a graph $G$ to various associated graphs of $G$ is studied in \cite{GS0}. The major concepts and findings in this paper are mentioned in this section. More over, we establish some new results in this area.

An obvious result about an IASI $f$ of a given graph $G$ is its hereditary nature. The following theorem shows that an IASI of a given graph induces an IASI on any subgraph of $G$.

\begin{proposition}\label{P-IASI2}
\cite{GS0} If a given graph $G$ admits an IASL (or IASI) $f$, then any subgraph $H$ of $G$ also admits an IASI $f^*$, the restriction of $f$ to $V(H)$. That is, the existence of an IASL for a graph $G$ is a hereditary property.
\end{proposition} 

If we replace an element (a vertex or an edge) of $G$ by another element (a vertex or an edge) which is not in $G$ to form an associated graph structure, say $G'$, it is customary that the common elements of the graph $G$ and the associated graph $G'$ preserve the same set-labels and the newly introduced elements assume the same set-label of the corresponding deleted elements of $G$. That is, the IASI defined on $G$ induces an IASI for the graph $G'$. The IASI, thus defined for the newly formed graph $G'$ is called an {\em induced IASI} of $G'$. 

This induction property of integer additive set-labeling holds not only for the subgraphs of an IASI graph $G$, but also for certain other graph structures associated with $G$ such as minors, topological reductions  etc. 

In the following results, we consider only the IASL of the associated graph which are induced from the IASL of the graph $G$ concerned. 

By an {\em edge contraction} in $G$, we mean an edge, say $e$, is removed and its two incident vertices, $u$ and $v$, are merged into a new vertex $w$, where the edges incident to $w$ each correspond to an edge incident to either $u$ or $v$. 

The existence of an IASI for the graph obtained by contracting finite edges of an IASL graph $G$ is established in the following theorem. 

\begin{proposition}\label{P-GEC}
\cite{GS0} Let $G$ be an IASI graph and let $e$ be an edge of $G$. Then, $G\circ e$ admits an IASI.
\end{proposition}

An undirected graph $H$ is called a {\em minor} of a given graph $G$ if $H$ can be formed from $G$ by deleting edges and vertices and by contracting edges.

\begin{proposition}
Let $G$ be an IASI graph and let $H$ be a  minor of $G$. Then, $H$ admits an IASI.
\end{proposition}
\noindent The proof of this result is an immediate consequence of Proposition \ref{P-IASI2} and \ref{P-GEC}.

Let $G$ be a connected graph and let $v$ be a vertex of $G$ with $d(v)=2$. Then, $v$ is adjacent to two vertices $u$ and $w$ in $G$. If $u$ and $v$ are non-adjacent vertices in $G$, then delete $v$ from $G$ and add the edge $uw$ to $G-\{v\}$. This operation is called an {\em elementary topological reduction} on $G$. 
Then, we have

\begin{proposition}\label{P-IASI3a}
\cite{GS0} Let $G$ be a graph which admits an IASI. Then any graph $G'$, obtained by applying finite number of elementary topological reductions on $G$, admits an IASI. 
\end{proposition}

A {\em subdivision} or an {\em expansion} of a graph is the reverse operation of an elementary topological reduction. Under this operation, we introduce a new vertex, say $w$, to an edge $uv$ in $G$. Thus, the vertex $w$ replaces the edge $uv$ and two edges $uw$ and $wv$ are introduced to $G-uv\cup \{w\}$. Two graphs $G$ and $G'$ are said to be {\em homeomorphic} if there is a graph isomorphism from some subdivision of $G$ to some subdivision of $G'$. 

The following theorem establishes the existence of an induced IASL for a graph $G'$ which is homeomorphic to a given IASL graph $G$.

\begin{proposition}\label{P-IASI3b}
Let $G$ be an IASL (or IASI) graph. Then, any subdivision $G'$ of $G$ admits an (induced) IASL (or IASI).
\end{proposition}
\begin{proof}
Let $f$ be an IASL defined on a given graph $G$ and let $u,v$ be two adjacent vertices in $G$. Introduce a new vertex $w$ on the edge $uv$ so that $uv$ is subdivided to two edges $uw$ and $wv$. Let $G'=[G-(uv)]\cup \{uw, vw\}$. Define a function $g:V(G')\to \mathcal{P}(\mathbb{N}_0)$ such that $g(v)= f(v)$ if $v\in V(G)$ and $g(w)=f^+(uv)$.
Clearly, $g$ is an IASL on $G'$. This result can be extended to a graph obtained by finite number of subdivisions on $G$. There fore, any subdivision of an IASL graph admits an (induced) IASL.
\end{proof}

Invoking Proposition \ref{P-IASI3b}, the admissibility of an IASL by a graph that is homeomorphic to a given IASL graph. 

\begin{proposition}
Let $G$ be an IASL (or IASI) graph and let $H$ be a  a graph that is homeomorphic to $G$. Then, $H$ admits an IASL (or IASI).
\end{proposition} 
\begin{proof}
Assume that the graphs $G$ and $H$ are homeomorphic graphs. Let $G'$ and $H'$ be the isomorphic subdivisions of $G$ and $H$ respectively. Let $f$ be an IASL defined on $G$. By Proposition \ref{P-IASI3b}, the subdivision $G'$ admits an IASL, say $f'$ induced by $f$. Since $G'\cong H'$, define a function $g:V(H')\to \mathcal{P}(\mathbb{N}_0)$ by $g(v')=f(v)$, where $v'$ is the vertex in $H'$ corresponding to the vertex $v$ of $G'$. That is, $g$ assigns the same set-labels of the vertices of $G'$ to the corresponding vertices of $H'$. Therefore, $g(V(H'))=f(V(G'))$. Since $H$ can be considered as a graph obtained by applying a finite number of topological reductions on $H'$, by Proposition \ref{P-IASI3a}, $g$ induces an IASL on $H$. Therefore, $H$ is an IASL graph.
\end{proof}

The existence of integer additive set-labelings for certain other graphs associated to the given graph $G$, which are not obtained by replacing some elements of $G$ by certain other elements not in $G$, have been established in \cite{GS0}. 

An important graph of this kind which is associated to a given graph $G$ is the line graph of $G$ which is defined as follows.

For a given graph $G$, its {\em line graph} $L(G)$ is a graph such that  each vertex of $L(G)$ represents an edge of $G$ and two vertices of $L(G)$ are adjacent if and only if their corresponding edges in $G$ incident on a common vertex in $G$. The following theorem establishes the existence of IASI by the line graph of a graph.

\begin{theorem}\label{T-IASI-LG}
\cite{GS0} If a graph $G$ admits IASI, say $f$, then its line graph $L(G)$ also admits an (induced) IASI.
\end{theorem}

Another graph that is associated to a given graph is its total graph which is defined as follows.

The {\em total graph} of a graph $G$ is the graph, denoted by $T(G)$, is the graph having the property that a one-to one correspondence can be defined between its points and the elements (vertices and edges) of $G$ such that two points of $T(G)$ are adjacent if and only if the corresponding elements of $G$ are adjacent (either  if both elements are edges or if both elements are vertices) or they are incident (if one element is an edge and the other is a vertex). 

The following theorem establishes the existence of an IASI for the total graph of a graph.

\begin{theorem}\label{T-IASI-TG}
\cite{GS0} If a graph $G$ admits IASI, say $f$, then its total graph $T(G)$ also admits an (induced) IASI.
\end{theorem}

So far, we discussed about the integer additive set-labeling of graphs with respect to the countably infinite set $\mathbb{N}_0$. Can we choose a finite set $X$ as the ground set for labeling the vertices of a graph $G$? This is possible only when the ground set $X$ has sufficiently large cardinality. Hence, the study about the cardinality of the ground set $X$ arises much interest.

The minimum cardinality of the ground set that is required for set-labeling a graph $G$ is called the {\em set-indexing number} of $G$. The following theorem determines the set-indexing number of a given graph $G$.  

\begin{theorem}\label{P-FinSet}
\cite{GS0} Let $X$ be a finite set of non-negative integers and let $f:V(G)\to 2^X-\{\emptyset\}$ be an IASI on $G$, which has $n$ vertices. Then, $X$ has at least $\lceil log_2(n+1)\rceil$ elements.
\end{theorem}

The following theorem is  special case of the above theorem when the vertices of $G$ have the same set-indexing number.

\begin{theorem}\label{P-UniSet}
\cite{GS0} Let $X$ be a finite set of non-negative integers and let $f:V(G)\to 2^X-\{\emptyset\}$ be an IASI on $G$, which has $n$ vertices, such that $V(G)$ is $l$-uniformly set-indexed. Then, $n\le \binom{|X|}{l}$.
\end{theorem}

We have discussed the cardinality of the ground set $X$ for labeling the vertices of a given graph so that it admits an IASL. The cardinality of the set-labels of the elements of $G$ is also worth studying. Some studies in this area are reviewed in the following section.
 
\section{Cardinality of the Set-Labels of IASL Graphs}

Certain studies on set-indexing numbers of the elements of the IASI-graphs have been done in \cite{GS1}, \cite{GS2} and \cite{GS0}. The set-theoretic foundations of the IASLs , established in these studies, are reviewed in this section.

Let $A$ and $B$ be two non-empty sets of non-negative integers. Then the ordered pairs $(a,b)$ and $(c,d)$ in $A\times B$ is said to be {\em compatible} if $a+b=c+d$. If $(a,b)$ and $(c,d)$ are compatible, then we write $(a,b)\sim (c,d)$. Clearly, $\sim$ is an equivalence relation. 

A {\em compatible class} with respect to an integer $k$ is the subset of $A\times B$ defined by $\{(a,b)\in A\times B:a+b=k\}$ and is denoted by $\mathsf{C}_k$. 

All the compatibility classes need not have the same number of elements but can not exceed a certain number. It can be noted that exactly one representative element of each compatibility class $\mathsf{C}_k$ of $f(u)\times f(v)$ contributes an element to the set-label of the corresponding edge $uv$ and all other $r_k-1$ elements are neglected and we call these elements {\em neglected elements} of the class $\mathsf{C}_k$. The number of neglected elements in the set-label of an edge $uv$ with respect to the set $f(u)\times f(v)$ is called the {\em neglecting number} of that edge.

The compatibility classes which contain the maximum possible number of elements are called {\em saturated classes}. The compatibility class that contains maximum elements is called a {\em maximal compatibility class}. Hence, we have

\begin{proposition}\label{P-CardCC}
\cite{GS0} The cardinality of a saturated class in $A\times B$ is $n$, where $n=\min(|A|,|B|)$. 
\end{proposition}

The number of distinct compatibility classes in $A\times B$ is called the {\em compatibility index} of the pair of sets $(A,B)$ and is denoted by $\mho_{(A,B)}$. Hence, we have the following lemma.

\begin{lemma}\label{L-3}
Let $f$ be an IASI of a graph $G$ and $u,v$ be two vertices of $G$. The set-indexing number of an edge $e=uv$ of $G$ is given by $|f^+(e)|=|f(u)|+|f(v)|=\mho_{(f(u),f(v))}$
\end{lemma}

The following result establishes the bounds for the set-indexing number of the edges of an IASI graph, in terms of the set-indexing numbers of their end vertices.

\begin{lemma}\label{L-0}
\cite{GS1} For an IASI $f$ of a graph $G$, $\max(|f(u)|,|f(v)|) \le |f^+(uv)|= |f(u)+f(v)| \le |f(u)| |f(v)|$, where $u,v \in V(G)$.
\end{lemma}

\begin{theorem}\label{T-Card}
Let $f$ be an IASI of a graph $G$ and let $u$ and $v$ be two adjacent vertices of $G$. Let $|f(u)|=m$ and $|f(v)|=n$. Then, $|f^+(uv)|=mn-r$, where $r$ is the number of neglected elements in $f(u)\times f(v)$. 
\end{theorem}

An interesting question that arises in this context is about a necessary and sufficient condition for an IASL of a given graph $G$ to be a uniform IASL of $G$. The following result provides a platform for finding the conditions for a given graph $G$ to admit a uniform IASL.

\begin{proposition}
Two adjacent edges $e_i$ and $e_j$ of a graph $G$ have the same set-indexing number if and only if the set-indexing number of the common vertex of these edges is the quotient when the difference of the neglecting numbers of the edges is divided by the differences of set-indexing numbers of the end vertices that are not common to these edges.  
\end{proposition}
\begin{proof}
Let $e_i=uv_i$ and $e_j=uv_j$ be two adjacent edges in the given graph $G$. Let $f$ be an IASL defined on $G$. Also, let $|f(u)|=m$, $|f(v_i)|=n_i$ and $|f(v_j)|=n_j$. Then, the set-indexing number of $e_i$ is $mn_i-r_i$ and that of $e_j$ is $mn_j-r_j$ where $r_i$ and $r_j$ are neglecting numbers of $e_i$ and $e_j$ respectively.

Assume that $e_i$ and $e_j$ have the same set-indexing number. Then, $mn_i-r_i=mn_j-r_j$. That is, $m=\frac{r_i-r_j}{n_i-n_j}$.
 
Conversely, assume that the set-indexing number of the common vertex of these edges is the quotient when the difference of the neglecting numbers of the edges is divided by the differences of set-indexing numbers of the end vertices that are not common to these edges. That is, $m=\frac{r_i-r_j}{n_i-n_j}$. Hence, $mn_i-r_i=mn_j-r_j$.

\noindent Therefore, the edges have the same set-indexing number.
\end{proof}

\noindent Due to the above proposition, we establish a necessary and sufficient condition for an IASL of a given graph to be a uniform IASL of $G$.

\begin{theorem}
Let $G$ be a graph that admits an IASL $f$. Then, $f$ is a uniform IASL of $G$ if and only if the set-indexing number of the common vertex of any two adjacent edges is the ratio of the difference between their neglecting numbers to the difference between the set-indexing numbers of their distinct end vertices.
\end{theorem}

In the following theorem, we prove a necessary and sufficient condition for an IASL of a graph $G$, which assigns uniform set-labels to the vertices of $G$, to be a uniform IASL of $G$.
 
\begin{theorem}
Let the graph $G$ admits an IASL $f$ under which all vertices of $G$ are uniformly set-labeled. Then, $f$ is a uniform IASL if and only if all edges of $G$ have the same neglecting number.
\end{theorem}
\begin{proof}
Let $V(G)$ be $l$-uniformly set-labeled. Then, for any vertex $u$ in $G$, $|f(u)|=l$ and $f^+(e_i)f^+(uv)=l^2-r_i$. 

Let $f$ be a uniform IASI of $G$. Then, for any two adjacent edges $e_i=uv_i$ and $e_j=uv_j$ of $G$, we have 
\begin{eqnarray*}
|f^+(e_i)|& = &|f^+(e_j)|\\
\implies l^2-r_i & = & l^2-r_j\\
\implies r_i-r_j & = & 0\\
\implies r_i & = & r_j.
\end{eqnarray*}

Conversely, assume that all the edges of $G$ have the same neglecting number, say $r$. Therefore, for any edge $e=uv$ of $G$, $|f^+(e)|=|f(u)|\,|f(v)|-r$. Since the graph $V(G)$ is $l$-uniformly set-labeled, $|f(v)|=l$ for all $v\in V(G)$. Therefore, $|f^+(e)|=l^2-r$ for all edges $e\in E(G)$. That is, $f$ is a uniform IASL.
\end{proof}

A necessary and sufficient condition for a graph to have a $2$-uniform IASI is proved in \cite{TMKA} as follows.

\begin{theorem}
\cite{TMKA} A graph $G$ admits a $2$-uniform IASI if and only if $G$ is bipartite.
\end{theorem}

In view of Lemma \ref{L-0} the following definitions are introduced in \cite{GS1} and \cite{GS2}.

\begin{definition}\label{D-WIASI}{\rm
\cite{GS1} A {\em weak IASI} is an IASI $f$ whose associated function is defined as $|f^+(uv)|=\max(|f(u)|,|f(v)|)$ for all $u,v\in V(G)$.   A graph which admits a weak IASI may be called a {\em weak IASI graph}. A weak  IASI is said to be {\em weakly uniform IASI} if $|f^+(uv)|=k$, for all $u,v\in V(G)$ and for some positive integer $k$.}
\end{definition}

\begin{definition}\label{D-SIASI}{\rm
\cite{GS2} An IASI $f$ is called a {\em strong IASI} if $|f^+(uv)|=|f(u)| |f(v)|$ for all $u,v\in V(G)$. A graph which admits a  strong IASI may be called a {\em strong IASI graph}. A  strong  IASI is said to be  {\em strongly uniform IASI} if $|f^+(uv)|=k$, for all $u,v\in V(G)$ and for some positive integer $k$.}
\end{definition}

Form the above definitions it follows that if $f;V(G)\to \mathcal{P}(\mathbb{N}_0)$ is a weak (or strong) IASI defined on a graph $G$, then for each adjacent pair of vertices $u$ and $v$ of $G$, each compatibility class of the pair of set-labels $f(u)$ and $f(v)$ is a trivial class.

Let $G$ be an IASI graph. Then, the following result for $G$ was proved in \cite{GS1}.

\begin{lemma}\label{L-IASI01}
\cite{GS1} Let $f:V(G)\to \mathcal{P}(\mathbb{N}_0)$ be an IASI defined on a graph $G$. Let $u$ and $v$ be any two adjacent vertices in $G$. Then, $|f^+(uv)|=\max(|f(u)|,|f(v)|)$ if and only if either $|f(u)|=1$ or $|f(v)|=1$.
\end{lemma}

This lemma leads us to the following result.

\begin{theorem}
An IASI $f$ of a graph $G$ is a weak IASI of $G$ if and only if at least one end vertex of any edge in $G$ has the set-indexing number $1$.
\end{theorem}

The following theorem describes a necessary and sufficient condition for a graph to admit a weakly uniform IASI.

\begin{theorem}
A connected graph $G$ admits a weakly uniform IASI $f$ if and only if $G$ is a bipartite graph.
\end{theorem}

The necessary and sufficient condition for the sum set of two sets to have maximum cardinality is established in the following result. 

\begin{lemma}\label{L-RDS}
\cite{GS2} Let $A$, $B$ be two non-empty subsets of $\mathbb{N}_0$. Let $D_A$ and $D_B$ be the sets of all differences between two elements in the sets $A$ and $B$ respectively. Then, $|A+B|=|A|.|B|$ if and only if  $D_A$ and $D_B$ are disjoint.
\end{lemma}

\noindent The sets $D_A$ of all differences between two elements of  set $A$ is called the {\em difference set} of $A$.

In view of the above lemma,  a necessary and sufficient condition for an IASI to be a strong IASI is established in \cite{GS2} as follows.

\begin{theorem}
An IASI $f$ of a graph $G$ is a strong IASI of $G$ if and only if the difference sets of the set-labels of any two adjacent vertices in $G$ are disjoint sets.
\end{theorem}

The following theorem describes a necessary and sufficient condition for a graph to admit a strongly uniform IASI.

\begin{theorem}
A connected graph $G$ admits a strongly $k$-uniform IASI $f$ if and only if $G$ is a bipartite graph or $f$ is a $(k,l)$-completely uniform IASI of $G$, where $k=l^2$.
\end{theorem}

\section{Conclusion}
So far, we have reviewed the studies about integer additive set-labelings and integer additive set-indexers of certain graphs and their characteristics. In this paper, we also propose some new results in this area. Certain problems regarding the admissibility of integer additive set-indexers by various other graph classes are still open.

More properties and characteristics of different types of IASLs and IASIs, both uniform and non-uniform, are yet to be investigated. There are some more open problems regarding  necessary and sufficient conditions for various graphs and graph classes to admit certain IASIs.

\section*{Acknowledgement}
The authors dedicate this work to the memory of Professor Belamannu Devadas Acharya who introduced the concept of integer additive set-indexers of graphs.

\end{document}